\documentclass[11pt]{amsart}
\usepackage{amsmath,amssymb,epsfig,color, amsthm, mathrsfs,mathbbol,comment}
\usepackage{wrapfig,lipsum,floatrow,longtable}%
\usepackage{graphicx}
\usepackage{multirow}
\usepackage{hhline}
\usepackage{url}
\graphicspath{ {images/} }
\allowdisplaybreaks

\newtheorem{theorem}{Theorem}[section]
\newtheorem{conjecture}[theorem]{Conjecture}
\newtheorem{question}[theorem]{Question}
\newtheorem{corollary}[theorem]{Corollary}
\newtheorem{definition}[theorem]{Definition}

\newtheorem{remark}[theorem]{Remark}

\newcommand{\vanish}[1]{}\parskip=12pt

\newcommand{\sign}{\mbox{\rm sign}}

\def\b{\qopname\relax o{braid}}
\def\br{\qopname\relax o{bridge}}
\def\R{\mathbb{R}}
\def\Z{\mathbb{Z}}

\numberwithin{equation}{section}

\newcommand{\Ebend}{E_{\mathrm{bend}}}

\begin{document}

\title{Knots with equal bridge index and braid index}

\author{Yuanan Diao$^\dagger$, Claus Ernst$^*$ and Philipp Reiter$^\sharp$}
\address{$^\dagger$ Department of Mathematics and Statistics\\
University of North Carolina Charlotte\\
Charlotte, NC 28223, USA}
\address{$^*$ Department of Mathematics\\
Western Kentucky University\\
Bowling Green, KY 42101, USA}
\address{$^\sharp$ Faculty of Mathematics\\
Chemnitz University of Technology\\
09107 Chemnitz, Saxony, Germany}

\begin{abstract}
In this paper we are interested in %
\emph{BB knots}, namely knots and links where the bridge index equals the braid index.
Supported by observations from experiments, it is conjectured that BB knots %
possess a special geometric/physical property
(and might even be characterized by it): if the knot is realized by a (closed) springy metal wire, then the equilibrium state of the wire is in an almost planar configuration of multiple (overlapping) circles. In this paper we provide a %
heuristic explanation to the conjecture and explore the plausibility of the conjecture numerically. We also identify BB knots among various knot families. For example, we are able to identify all BB knots in the family of alternating Montesinos knots, as well as some BB knots in the family of the non-alternating Montesinos knots, and more generally in the family of the Conway algebraic knots. The BB knots we identified in the knot families we considered include all of the 182 one component BB knots with crossing number up to 12. Furthermore, we show that the number of BB knots with a given crossing number $n$ grows exponentially with $n$. 
\end{abstract}

\maketitle

\section{Introduction}

In his famous paper \cite{milnor}, Milnor showed how a knot invariant may be related to geometric properties of the space curve representing a knot by providing a lower bound of the
total curvature of an embedded closed curve explicitly in terms of its bridge index of the knot the curve represents. In a similar spirit, Langer and Singer~\cite{LS} raised a question concerning the equilibrium shape (a physical property) of a closed springy wire and the knot invariants of the knot the closed wire represents. Gallotti and Pierre-Louis~\cite{gallotti} conjectured that the equilibrium shapes within knot types $K$ where braid index $\b(K)$ and
bridge index $\br(K)$ coincide, will be the $\br(K)$-times covered circle. This conjecture has recently been proven in the case of 2-bridge knots~\cite{GRvdM}. For the sake of convenience, we shall call a knot type  whose bridge index and braid index coincide a {\em BB knot}.

Heuristically, it is plausible that the equilibrium shape
of a very thin springy wire should be, if topologically possible,
very close to a $k$-times covered circle.
Recalling the F\'ary--Milnor inequality~\cite{milnor},
the minimum value would be
$k=\br(K)$.
On the other hand, a configuration of a wire that passes
$\br(K)$-times around a circle constitutes in
fact a braid presentation which requires that $k\ge \b(K)$
which by $\b(K)\ge\br(K)$ implies $\br(K)=\b(K)$. The left of Figure \ref{experiment} shows an experiment of the equilibrium state of the figure eight knot realized by a closed springy wire, which is clearly not close to a multiply covered circle. Notice that the figure eight knot is not a BB knot since its bridge index is 2 and braid index is 3. On the other hand, our numerical simulations
indicate that an equilibrium state of a BB knot $K$ is
close to a $\br(K)$-times covered circle as shown in Figure \ref{fig:sims}.

The class of BB knots might also be of potential interest to scientists looking for potential candidates of knot types that can be constructed via molecular knots.
In~\cite{micheletti} the authors give a list of knots that either have been constructed or could potentially be constructed as different molecular knots. At least in the symmetric cases, we observe that there seems to be a prevalence for BB knots, in particular among knot classes with larger crossing number.

In this light, it is a natural question to investigate the
family of all knot or link classes $K$ with $\br(K)=\b(K)$. This is the main focus of this paper.
As a first step towards a complete classification (which might or might
not be feasible) we searched the database KnotInfo \cite{knotinfo} and found that there are 182 one component BB knots with crossing number up to $12$ (they are listed in Table~\ref{table}).
Next, we consider certain families of classes
(namely torus knots/links, 2-bridge knots/links, Montesinos links,
Conway algebraic knots)
and identify (infinite) subfamilies of BB knots within. 

Moreover, there are a number of interesting questions
concerning the number $\mathcal B_{n}$ of BB knots with crossing number $n$.
Does $\mathcal B_{n}$ grow exponentially?
What about the ratio of $\mathcal B_{n}$ over the number of \emph{all} knots
with with crossing number $n$~?

\begin{figure}[!h]
\label{experiment}
\includegraphics[scale=.18,trim=20 0 180 0,clip]{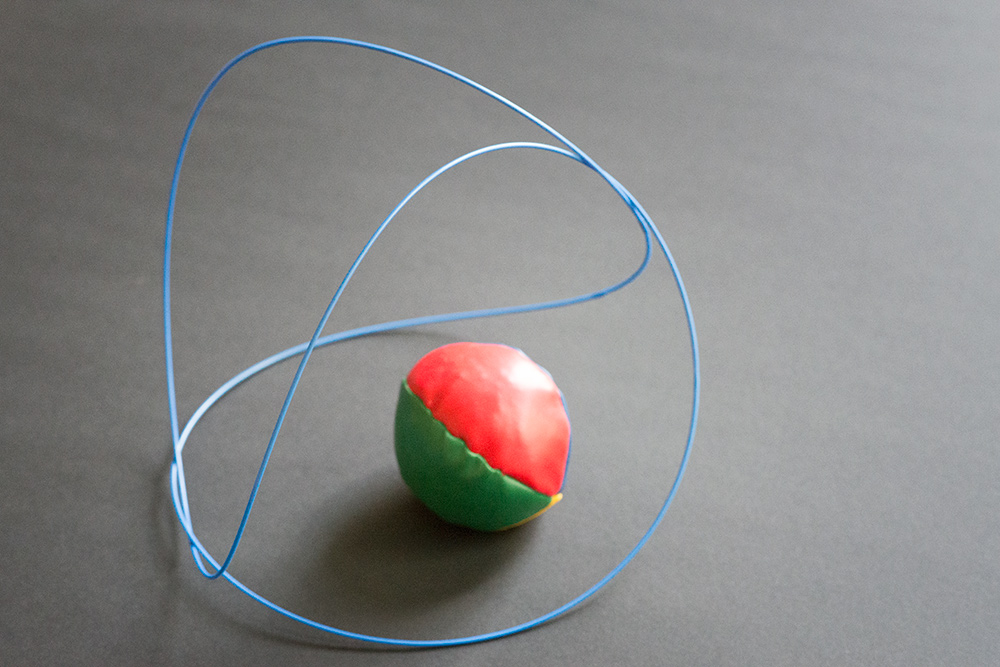}
\quad%
\rotatebox[origin=c]{-20}{\reflectbox{\includegraphics[scale=.36]{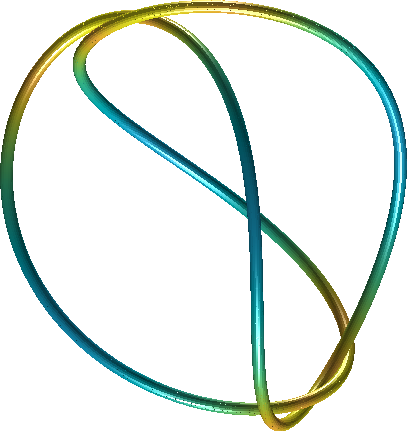}}}
\caption{Left: An experiment with a springy wire belonging to the figure-eight class (Wire model manufactured by \textsc{why knots}, Aptos, in 1980, photographed by Bernd Bollwerk, Aachen);
Right: A  local minimizer produced by numerical simulation featuring nearly the same shape (see~\cite{BR} for details).
}
\end{figure}

Before we proceed to the next section, we need to clarify a couple of key terms and concepts used in this paper. First,  in what follows we will use the word ``knot'' to indicate both a knot type or a particular knot embedding. What we mean will depend on the context and should be clear to the reader. We will use the term ``knot class'' when we talk about the set of embeddings of a knot type with certain geometric restrictions.
Moreover, the term ``knot" in this paper is also used for the commonly used term ``link" in other literature. That is, a knot in this paper may have more than one component. When there is a need, we shall make this clear by specifying the number of components in the knot being discussed. For example, the BB knots listed in Table \ref{table} are all knots with one component. Second, in the definition of BB knots, the knots are un-oriented (since orientation plays no role to the physical property of interest that leads to the definition of BB knots). However our approach to this problem requires the use of the braid index defined for oriented knots. We would like to point out the connection between these two definitions to avoid confusion. Let $K$ be an un-oriented knot, $\vec{K}$ be an oriented knot corresponding to $K$ (that is, each component in $K$ has been assigned an orientation) and $\b(\vec{K})$ be the braid index of $\vec{K}$ (as an oriented knot). Then un-oriented braid index $\b(K)$ of $K$, which is what we have used in the above in the definition of BB knots, is defined as the minimum of $\vec{K}$ where $\vec{K}$ runs over all possible orientation assignments to the components of $K$. It is known (a rather obvious fact) that $\b(\vec{K})\ge \br(K)$. Thus if $\b(\vec{K})= \br(K)$ for some $\vec{K}$,  then we must have $\b(\vec{K})= \b(K)$. Of course, if $k$ has only one component then for any choice of orientation $\b(\vec K)=\b (K)$.

We organize the rest of the paper as follows: In Sections \ref{sec:bend} and \ref{numerics} we discuss BB knots made of thin springy wires (so called \emph{elastic knots}) and consider the  bending energy of elastic knots in a numerical experiment. 
In Section \ref{knotfamilies} we identify BB knots among several knot families. In Section \ref{numberofknots} we show that the number of BB knots with a given crossing number grows exponentially as a function of the crossing number
and state a few questions.

\section{Bending energy minimizers within knot classes}\label{sec:bend}

The aim of this section is to explain how the BB knots appear in an elementary model
for the experiment described in the introduction.
For convenience, we will present here a simplified version
of the model of \emph{elastic knots} which is described in Section~\ref{numerics}
below.

\subsection{A variational problem}

Neglecting all other physical forces such as friction, shear, and twist,
we will assume that the behavior of that springy wire is only affected
by the bending energy of its centerline $\gamma:\R/\Z\to\R^{3}$
\[ \Ebend(\gamma)=\int_{\gamma}(\kappa(s))^{2} ds \]
where $\kappa(s)$ denotes the curvature of $\gamma$ at $s$
while $s$ is an arc-length parameter. Without loss of generality
we may assume that $\gamma$ is parametrized by arc-length which gives
$\Ebend(\gamma)=\|\gamma''\|_{L^{2}}^{2}$.

Our aim would be to find global minimizers of $\Ebend$ within a given knot class $K$.
As $\Ebend$ is not invariant under scaling, we will consider only unit length embeddings
which leads to considering the set
\[ \mathscr C(K) = \left\{\gamma\in W^{2,2}(\R/\Z,\R^{3}) \middle|
|\gamma'(s)|\equiv1, \gamma(0)=0_{\R^{3}},\gamma\in K\right\}. \]
Here we fix $\gamma(0)$ just for technical reasons.

We will always assume that $K$ is a non-trivial tame knot class,
for the other cases are uninteresting.
The global minimizer of $\Ebend$ within the unknot class is
just the circle.
As wild knots are not $C^{1}$ and $W^{2,2}\hookrightarrow C^{1,1/2}$
by the Sobolev embedding theorem, $\mathscr C(K)=\emptyset$ for wild $K$.

\subsection{Minimizers}

It turns out that, except for the trivial knot class, there are no such minimizers:
Any potential minimizer within $K$ would be $C^{1}$ (due to the Sobolev embedding
$W^{2,2}\hookrightarrow C^{1,1/2}$) and embedded (since it is a knot).
According to~\cite{DEJvR} %
there is a $C^{1}$-neighborhood consisting of knotted curves belonging to $K$ as well.
In particular, this neighborhood contains a $W^{2,2}$-neighborhood
with the same property which demonstrates that $\gamma$ is not only
a global minimizer of $\Ebend$ within $K$ but also a local minimizer
of $\Ebend$ with respect to \emph{all} curves.
Now, by Langer and Singer's seminal result~\cite{LS},
the circle is the only curve with the latter property.
This fact reflects the observation that there are indeed points of
self-contact present in physical models.

Nevertheless we can consider a minimal sequence $(\gamma_{k})_{k\in\mathbf N}$ with respect to $\Ebend$
in $\mathscr C(K)$,
{\em i.e.}, $\gamma_{k}\in\mathscr C(K)$ for all $k\in\mathbf N$ with $\Ebend(\gamma_{k})\to\inf_{\mathscr C(K)}\Ebend$ as $k\to\infty$.
By the fact that $\Ebend(\gamma)=\|\gamma''\|_{L^{2}}^{2}$
and $\gamma(0)=0_{\R^{3}}$ for all $\gamma\in\mathscr C(K)$
we deduce that this minimal sequence is uniformly bounded in $W^{2,2}$.
Being a reflexive space, bounded sets in $W^{2,2}$
are weakly compact.
Therefore we may extract a subsequence converging to some limit curve $\gamma_{0}$
with respect to the weak $W^{2,2}$-topology.
Any such limit curve belongs to the weak $W^{2,2}$-closure
of $\mathscr{C}(K)$ which will be denoted by $\overline{\mathscr{C}(K)}$.

Potentially there can be many of those limit curves, depending on the
minimal sequence.
Any of these limit curves has the following properties:
It is not embedded (due to~\cite{LS}, unless $K$ is trivial), so it
belongs to the weak $W^{2,2}$-\emph{boundary} of $\mathscr C(K)$.
Its bending energy provides
a lower bound on the bending energy of any curve in $\mathscr C(K)$.
This is due to the fact that $\Ebend$ is lower \emph{semi}continuous with respect to weak $W^{2,2}$-convergence.

\subsection{The case of BB knots}

Recall that the F\'ary--Milner inequality~\cite{milnor} bounds the
total curvature $\int\kappa(s)ds$ of a curve belonging to $K$ by $2\pi \br(K)$ such that
we have for $\gamma\in\mathscr C(K)$
\[ 2\pi \br(K)<\int_{\gamma}\kappa(s)ds=\|\gamma''\|_{L^{1}}. \]
As $\gamma\mapsto{}\|\gamma''\|_{L^{1}}$ is not continuous with respect to weak
$W^{2,2}$-convergence it is not clear whether this lower bound
also holds for $\overline{\mathscr C(K)}$.
In fact, using the existence of \emph{alternating} quadrisecants
estabished by Denne~\cite{denne}, this has been
proven~\cite[Appendix]{GRvdM} for $\br(K)=2$. In the following
we \emph{assume} that it also holds for higher-order bridge indices.
Invoking the Cauchy--Schwarz inequality, the limit curve $\gamma_{0}\in\overline{\mathscr C(K)}$ satisfies
\[ 2\pi \br(K)\le\int_{\gamma_{0}}\kappa(s)ds=\|\gamma_{0}''\|_{L^{1}}
\le\|\gamma_{0}''\|_{L^{2}} = \left(\Ebend(\gamma_{0})%
\right)^{1/2}. \]
Using a braid representation we can construct
$C^{2}$-smooth (even $C^{\infty}$-smooth) curves belonging to $K$ inside the standard
$\varrho$-torus in $\R^{3}$, {\em i.e.}, the uniform neighborhood of
width $\varrho\in(0,1)$ of the unit circle,
such that each disk of the $\varrho$-torus is intersected $\b(K)$-times.
Translating and rescaling,
we obtain a family $(\beta_{\varrho})_{\varrho\in(0,1)}\subset\mathscr C(K)$ such that 
$\beta_{\varrho}$ converges (with respect to the $W^{2,2}$-norm) to a $\b(K)$-times
covered circle of length one as $\varrho\searrow0$.
Of course, the latter has bending energy $(2\pi \br(K))^{2}$.
As $\gamma_{0}$ is an $\Ebend$-minimizer within $\overline{\mathscr C(K)}$,
we arrive, for any $\varrho\in(0,1)$, at
\[ (2\pi \br(K))^{2}\le\left(\int_{\gamma_{0}}\kappa(s)ds\right)^{2}=\|\gamma_{0}''\|_{L^{1}}^{2}
\le\|\gamma_{0}''\|^2_{L^{2}} = \Ebend(\gamma_{0}) \le \Ebend(\beta_{\varrho})
\xrightarrow{\varrho\searrow0} (2\pi \b(K))^{2}. \]
Here the condition $\b(K)=\br(K)$ comes into play for it implies that
all terms in the previous line are equal. Equality in the Cauchy--Schwarz inequality implies
that the integrand is constant a.e. Therefore the curvature of the minimizer
$\gamma_{0}$ must be equal to $2\pi \br(K)$ a.e.

\subsection{Caveat}

It is important to note that the latter condition does \emph{not}
imply that $\gamma_{0}$ is a $\br(K)$-times covered circle.
Indeed, that would be wrong.
In case $\br(K)=2$ one can rigorously prove~\cite{GRvdM} that
any $\Ebend$ minimizer $\gamma_{0}$ within $\mathscr C(K)$
consists of two circles that either coincide or tangentially meet in
precisely one point.
Up to isometries and reparametrization, the set of those minimizers can be
parametrized by the angle between the two circles.
In the general situation we would expect a number of $\br(K)$ circles
tangentially meeting in (at least) one point.

However, adding a positive thickness restriction to $\mathscr C(K)$ or a penalty term to $\Ebend$
acts like a choice criterion that selects the $\br(K)$-times covered circle from that family of $\Ebend$-minimizers.
This will be outlined below in Section~\ref{numerics}.

\section{BB knots realized by elastic wires}\label{numerics}

\subsection{Elastic knots}

In Section~\ref{sec:bend} above we considered the problem
to minimize $\Ebend$ in the class $\mathscr C(K)$ of curves
in the knot class~$K$. In order to do so, we looked at a
minimal sequence of curves and extracted a subsequence
that converges to some limit curve $\gamma_{0}$.
The problem is that these limit curves are not unique.
For instance, in the case of 2-bridge torus knots ({\em i.e.}, the knots $T(n,2)$) we
obtain an entire one-parameter family of minimizers.
Which of them is the most ``realistic'' one,
{\em i.e.}, one that would be observed in physical experiments?

To answer this question,
we need to incorporate the fact that we are looking at
physical ropes that have some very small thickness in our simulation model.
This can be done by either imposing a constraint to the space of curves~\cite{gallotti,vdM:eke3}, 
or by adding a penalty term to the bending energy~\cite{sossinsky,GRvdM,vdM:meek}.
We will discuss the latter approach, {\em i.e.}, we shall adopt the following model
\[ E_{\vartheta} = \Ebend+\vartheta\mathcal R, \qquad\vartheta>0, \]
where $\mathcal R$ denotes a self-avoiding functional
such as the ropelength, {\em i.e.}, the quotient of length over thickness~\cite{Buck,DEJvR,GM99}.
Now the minimization problem is well defined~\cite{GRvdM,vdM:meek},
{\em i.e.}, for any $\vartheta>0$ there is a curve $\gamma_{\vartheta}\in\mathscr C(K)$
such that $E_{vartheta}(\gamma_{\vartheta})=\inf_{\mathscr C(K)} E_{\vartheta}$.
In particular, for $\vartheta\in(0,1)$ we have that
\[ \|\gamma_{\vartheta}''\|_{L^{2}}^{2}
\le E_{\vartheta}(\gamma_{\vartheta})
= \inf_{\mathscr C(K)} E_{\vartheta}
\le E_{\vartheta}(\gamma_{1})
\le \Ebend(\gamma_{1}) + \mathcal R(\gamma_{1}) \]
is uniformly bounded, %
so we can extract a subsequence as detailed in Section~\ref{sec:bend}.
The limit curve $\gamma_{0}$ is referred to as an \emph{elastic knot} for the knot class $K$~\cite{GRvdM,vdM:meek}.
Note that, unless $K$ is trivial, $\gamma_{0}$ is not embedded.

There are other names for related concepts in the literature, namely
``stiff knot'' (Gallotti and Pierre-Louis~\cite{gallotti}),
and ``normal form'' (Sossinsky who even thought of a complete classification of
knot classes by this concept, see~\cite{sossinsky} and references therein).

Note that, a priori, we cannot expect either the minimizers
$(\gamma_{\vartheta})_{\vartheta>0}$ or the elastic knot $\gamma_{0}$
to be unique. A posteriori, one can show that the elastic knot
for the 2-bridge torus knot classes is (up to isometries) the doubly covered circle~\cite{GRvdM}.
The proof relies on a generalization of the crookedness estimate in Milnor's proof~\cite{milnor}.
Currently there are no rigorous results concerning elastic knots
for other knot classes. However, the $\b(K)$-times covered circle
is a candidate for an elastic knot for any BB class~$K$.
Vice versa, one might speculate that the BB classes
are the only ones whose elastic knots are several times
covered circles.

\subsection{Simulation of elastic knots}

Physical and numerical simulations related to elastic knots
have been carried out so far
by Avvakumov and Sossinsky~\cite{sossinsky}, Gallotti and Pierre-Louis~\cite{gallotti},
 Gerlach et al.~\cite{GRvdM}, as well as by one of the authors and colleagues recently in~\cite{BR,BRR}.
The recently launched web application \textsc{knotevolve}~\cite{knotevolve} allows for carrying out a large variety of new experiments.

Numerically, it is a challenging problem to approximate elastic knots as we face two forces which push the elastic knots in different directions. The experiments carried out in~\cite{BR} shed some light on the energy landscape.
For instance, it is unlikely that the configuration shown in
Figure~\ref{experiment} is an elastic knot for the figure-eight class; see also the discussion in~\cite[Section~6.3]{GiRvdM}.

In Figure~\ref{fig:sims} we show some simulation results for BB knots with braid index up to 3. To this end, we applied the algorithm introduced in~\cite{bartels13,BRR,BR,knotevolve}
to the  regularization parameter $\vartheta{}=10^{-4}$
and initial configurations from the Knot Server~\cite{knotserver}.
For the BB knots with braid index 3, these simulation results still show visible deviation from the
three-times covered circle, in spite of 
the relatively small regularization parameter.

\newcommand{\knot}[3][.2]{%
\fbox{\begin{minipage}[b][38mm]{.22\textwidth}
\includegraphics[scale=#1]{bribra/#2_#3.png}\par\vfill
\includegraphics[scale=#1]{bribra/#2_#3f.png}
\end{minipage}}\makebox[0cm][r]{\raisebox{1ex}{$#2_{#3}$\ }}}
\newcommand{\sknot}[3][.15]{%
\fbox{\begin{minipage}[b][25mm]{.17\textwidth}
\includegraphics[scale=#1]{bribra/#2_#3.png}\par\vfill
\includegraphics[scale=#1]{bribra/#2_#3f.png}
\end{minipage}}\makebox[0cm][r]{\raisebox{1ex}{$#2_{#3}$\ }}}

\begin{figure}
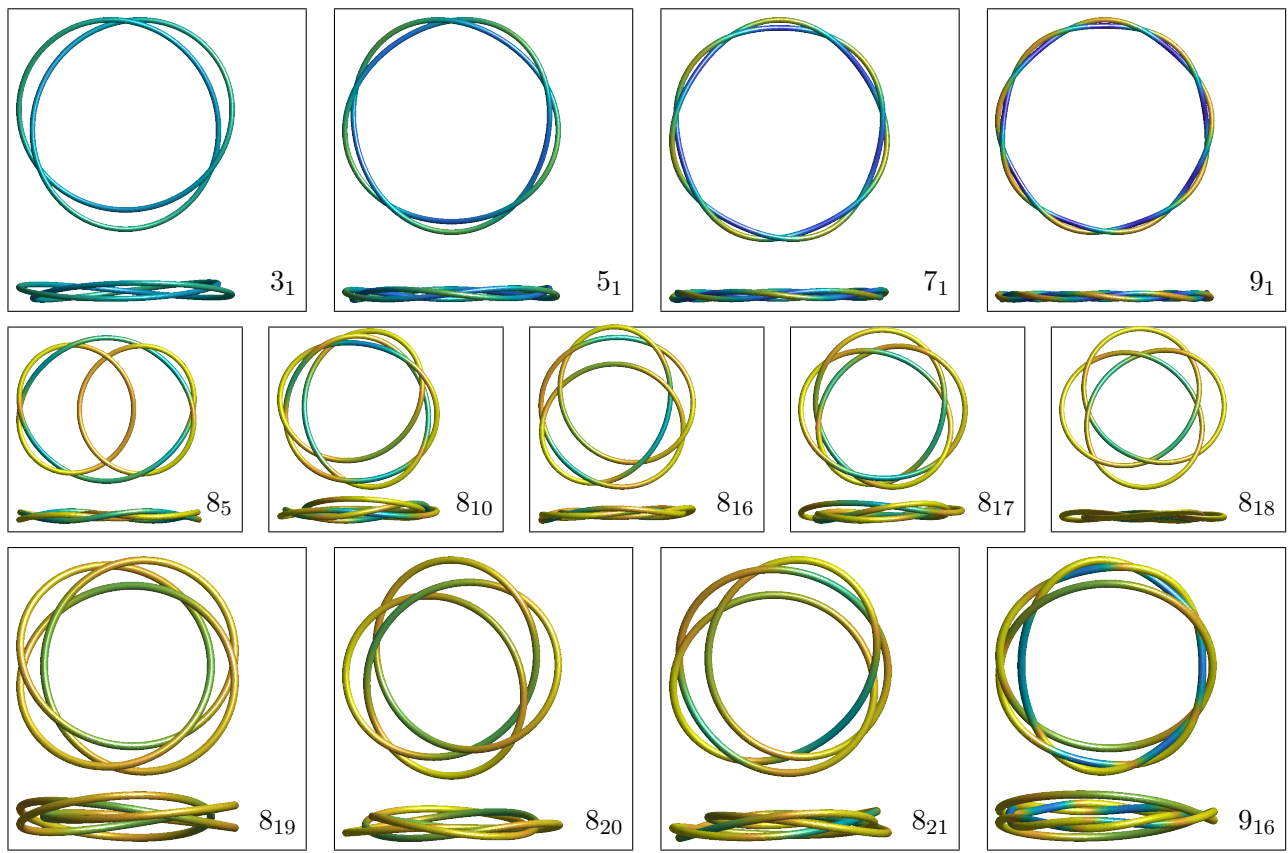
%
\knot3{1}\hfill
\knot5{1}\hfill
\knot7{1}\hfill
\knot9{1}\medskip

\sknot[.19]8{5}\hfill
\sknot8{10}\hfill
\sknot[.19]8{16}\hfill
\sknot[.17]8{17}\hfill
\sknot8{18}\medskip

\knot8{19}\hfill
\knot8{20}\hfill
\knot8{21}\hfill
\knot9{16}

\caption{Approximations of elastic knots for all thirteen BB classes
with crossing number at most nine.
Each curve is shown from top and front view;
colors correspond to local curvature.
The knots in the first row as well as $8_{19}$ are torus knots.}\label{fig:sims}

\end{figure}

\section{BB knot identification in knot families}
\label{knotfamilies}

In this section we will consider knot families including the torus knots, the 2-bridge knots, the alternating Montesinos knots, the non-alternating Montesinos knots, and the Conway algebraic knots. For the torus knots, the 2-bridge knots and the alternating Montesinos knots  we are able to identify all knots within these families that have equal bridge index and braid index. For the non-alternating Montesinos knots and the Conway algebraic knots, we have some partial results. The results for the torus knots and the 2-bridge knots are known, we decide to state them here for the sake of completeness. We would like to point out  that the  proof for the case of 2-bridge knots is new.

\subsection{Torus knots} \label{A}
Every chiral pair of torus knots has a representative that can be presented by a pair of positive integers $p$, $q$ such that $p\ge q\ge2$ and is denoted by $T(p,q)$, with the number of components in  $T(p,q)$ given by $\gcd(p,q)$.  It is well known that $\b(T(p,q))=\br(T(p,q))=q$ \cite{Mu2,S2} or \cite{Sch2007} for a more recent proof. That is, every torus knot has equal bridge index and braid index. It is known that the crossing number of $T(p,q)$ is $(q-1)p$ \cite{Mu2} hence the number of torus knots with a given crossing number $n$ is at most of the order of $n$. In other words, the number of torus knots will not contribute to the exponential growth of knots with equal bridge index and braid index. It is worthwhile for us to point out that when a torus knot has more than one component, it is assumed that all components are assigned parallel orientations. One component BB torus knots with crossing number up to 12 are listed in Table \ref{table} marked with a superscript $^t$. 

\subsection{ 2-bridge knots} 

This case is easy to deal with, since the only knots that can arise a 2-braids are the $T(n,2)$ torus knots. Thus we have by default the following theorem:

\begin{theorem}\label{torustwobridge} 
Let $K=B(\alpha,\beta)$ be a 2-bridge knot. Then $K$ is a BB knot if and only if $K$ is a  $T(n,2)$ torus knot.
\end{theorem}

In the following we give a second proof of the above theorem, introducing a method that will help us to deal with Montesinos knots in the next subsection and will be used in the proof of Theorem \ref{exp_thm}. To explain this method we need to describe the family of 2-bridge knots in more detail.
It is known that every 2-bridge knot can be represented by an alternating diagram associated with two co-prime positive integers $0<\beta<\alpha$ in the following way. A vector $(a_1,a_2,...,a_n)$
is called a {\it standard continued fraction decomposition} of $\frac{\beta}{\alpha}$ if $n$ is odd and all $a_i>0$ and
$$
\frac{\beta}{\alpha}=\frac{1}{a_1+\frac{1}{a_{2}+\frac{1}{.....\frac{1}{a_n}}}}. 
$$
It may be necessary to allow $a_n=1$ in order to guarantee that the length of the vector $(a_1,a_2,...,a_n)$ is odd. Under these conditions the standard continued fraction expansion of $\frac{\beta}{\alpha}$ is unique. An oriented 2-bridge knot (also called a {\em 4-plat} or a {\em rational knot}), denoted by $\vec{K}=\vec{B}(\alpha,\beta)$, is then presented by the {\em standard diagram} as  shown in Figure \ref{2bridgeone} using the vector $(a_1,a_2,...,a_n)$. Furthermore, without loss of generality for a standard diagram we can assign the component corresponding to the long arc at the bottom of Figure \ref{2bridgeone} the orientation as shown, since 2-bridge knots are known to be invertible. When a 2-bridge knot has two components, there are two choices for the orientation of the other component. Usually the two different orientations of the other component lead to two different oriented 2-bridge knots \cite{BZ, cromwell2004}  which may have different braid indices. For example, the braid index of the two-bridge link in Figure \ref{2bridgeone} is ten, however if we re-orient one of the two components the braid index changes to nine, see Theorem \ref{2bridge_theorem} below.

\begin{figure}[htb!]
\includegraphics[scale=1]{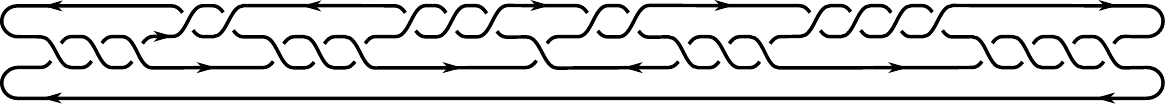}
\caption{The 2-bridge knot $\vec{B}(17426,5075)$ given by $(3,2,3,3,1,2,3,4,4)$.}
\label{2bridgeone}
\end{figure}

Since all crossings corresponding to a given $a_i$ have the same crossing sign under the given orientation, we will define a signed vector $(b_1,b_2,...,b_n)$ where $b_i =\pm a_i$ with its sign given by the crossing sign of the crossings corresponding to $a_i$. For example, for $\vec{K}=\vec{B}(17426,5075)$ with the orientation shown in Figure \ref{2bridgeone} we obtain the signed vector $(3,2,3,3,-1,-2,-3,4,-4)$. In \cite{DEHL2018} the following theorem was established:

\begin{theorem}\label{2bridge_theorem} \cite{DEHL2018} 
Let $\vec{K}=\vec{B}(\alpha,\beta)$ be an oriented 2-bridge link diagram with signed vector  $(b_1,b_2,...,b_{2k+1})$ in the standard form, then its braid index is given by
\begin{equation}\label{2bridgeformula}
\textbf{b}(K)=1+\frac{2+\sign(b_1)+\sign(b_{2k+1})}{4}+\sum_{b_{2j}>0,1\le j\le k}\frac{b_{2j}}{2}+\sum_{b_{2j+1}<0,0\le j\le k}\frac{|b_{2j+1}|}{2}.
\end{equation}
\end{theorem}

The above theorem allows the computation of the braid index of any 2-bridge link using a minimal diagram. The computation of the braid index of an oriented 2-bridge link was first given by Murasugi \cite{Mu} using a different method depending on a continued fraction expansion of $\beta/\alpha$ using only even integers. That the two methods are equivalent was shown in \cite{DEHL2018}. The formulation of Theorem \ref{2bridge_theorem} allows us to prove Theorem \ref{torustwobridge} in a way that can be generalized to Montesinos knots, see the next subsection.

\begin{proof}
The only if part is trivially true since every torus knot has equal bridge index and braid index as we discussed in Subsection \ref{A}. Now, if $K=B(\alpha,\beta)$ is a BB knot, then we must have $\b(K)=2$ hence there exists an oriented version $\vec{K}=\vec{B}(\alpha,\beta)$ such that $\b(\vec{K})=2$. By (\ref{2bridgeformula}) we have 
$$
2=1+\frac{2+\sign(b_1)+\sign(b_{2k+1})}{4}+\sum_{b_{2j}>0,1\le j\le k}\frac{b_{2j}}{2}+\sum_{b_{2j+1}<0,0\le j\le k}\frac{|b_{2j+1}|}{2},
$$
where  $(b_1,b_2,...,b_{2k+1})$ is the signed vector of $\vec{K}=\vec{B}(\alpha,\beta)$. Notice that $\sign(b_1)=\sign(b_2)$ if $b_2\not=0$. Furthermore, if $b_3>0$ then $b_2$ is even hence $b_2>1$. A proof of this can be found in \cite{DEHL2018}, a reader can also prove this directly by considering the Seifert circle decomposition of $\vec{K}$. We will prove the theorem in two separate cases:  $\sign(b_1)=1$ and $\sign(b_1)=-1$. 

Case 1. $\sign(b_1)=1$. If $\sign(b_{2k+1})=1$, then
$$
0=\sum_{b_{2j}>0,1\le j\le k}\frac{b_{2j}}{2}+\sum_{b_{2j+1}<0,0\le j\le k}\frac{|b_{2j+1}|}{2}\ge \frac{b_{2}}{2}\ge 0,
$$
hence $b_2=0$. So $k=0$ and $K=T(n,2)$ for some integer $n\ge 2$. On the other hand, if $\sign(b_{2k+1})=-1$, then 
\begin{eqnarray*}
&&1+\frac{2+\sign(b_1)+\sign(b_{2k+1})}{4}+\sum_{b_{2j}>0,1\le j\le k}\frac{b_{2j}}{2}+\sum_{b_{2j+1}<0,0\le j\le k}\frac{|b_{2j+1}|}{2}\\
&\ge &
2+\sum_{b_{2j}>0,1\le j\le k}\frac{b_{2j}}{2}+\sum_{b_{2j+1}<0,0\le j\le k-1}\frac{|b_{2j+1}|}{2}> 2,
\end{eqnarray*}
which is a contradiction hence this case is not possible.

Case 2. $\sign(b_1)=-1$. If $b_2=0$, then 
\begin{eqnarray*}
2&=&1+\frac{2+\sign(b_1)+\sign(b_{2k+1})}{4}+\sum_{b_{2j}>0,1\le j\le k}\frac{b_{2j}}{2}+\sum_{b_{2j+1}<0,0\le j\le k}\frac{|b_{2j+1}|}{2}\\
&= &
1+\frac{|b_{1}|}{2},
\end{eqnarray*}
hence $b_1=-2$ and $K$ is the Hopf link that is the mirror image of $T(2,2)$ discussed in Case 1 if we ignore the orientation. If $b_2\not=0$, then $b_2<0$ and it is necessary that $b_1=-1$ and $b_3<0$ in this case (again this can be observed by considering the Seifert circle decomposition of $K$). It follows that 
\begin{eqnarray*}
2&=&1+\frac{2+\sign(b_1)+\sign(b_{2k+1})}{4}+\sum_{b_{2j}>0,1\le j\le k}\frac{b_{2j}}{2}+\sum_{b_{2j+1}<0,0\le j\le k}\frac{|b_{2j+1}|}{2}\\
&\ge &
1+\frac{1+\sign(b_{2k+1})}{4}+\frac{1+|b_{3}|}{2}.
\end{eqnarray*}
Thus we must have $b_3=-1$ and $\sign(b_{2k+1})=-1$. However $b_3=-1$ implies either $b_4=0$ or $b_4>0$. Since $b_4>0$ leads to $\b(\vec{K})>2$, we must have $b_4=0$, hence $k=1$ and the signed vector of $K$ is of the form $(-1,-(n-2),-1)$ for some $n>2$. This is the mirror image of the torus knot $T(n,2)$ discussed in Case 1 above if we ignore the orientation.
\end{proof}

\subsection {Alternating Montesinos knots} We now consider the set of all alternating Montesinos knots, a family that is much larger than the family of 2-bridge knots. In general, a Montesinos knot $K=M(\beta_1/\alpha_1,\ldots, \beta_k/\alpha_s,\delta)$ is a knot with a diagram as shown in Figure \ref{Montesinos}, where each diagram within a topological circle (which is only for the illustration and not part of the diagram) is a rational tangle $A_j$ that corresponds to some rational number $\beta_j/\alpha_j$ with $|\beta_j/\alpha_j|<1$ and $1\le j\le s$ for some positive integer $s\ge 2$, and $\delta$ is an integer that stands for an arbitrary number of half-twists, see Figure \ref{Montesinos}.

\begin{figure}[htb!]
\includegraphics[scale=.4]{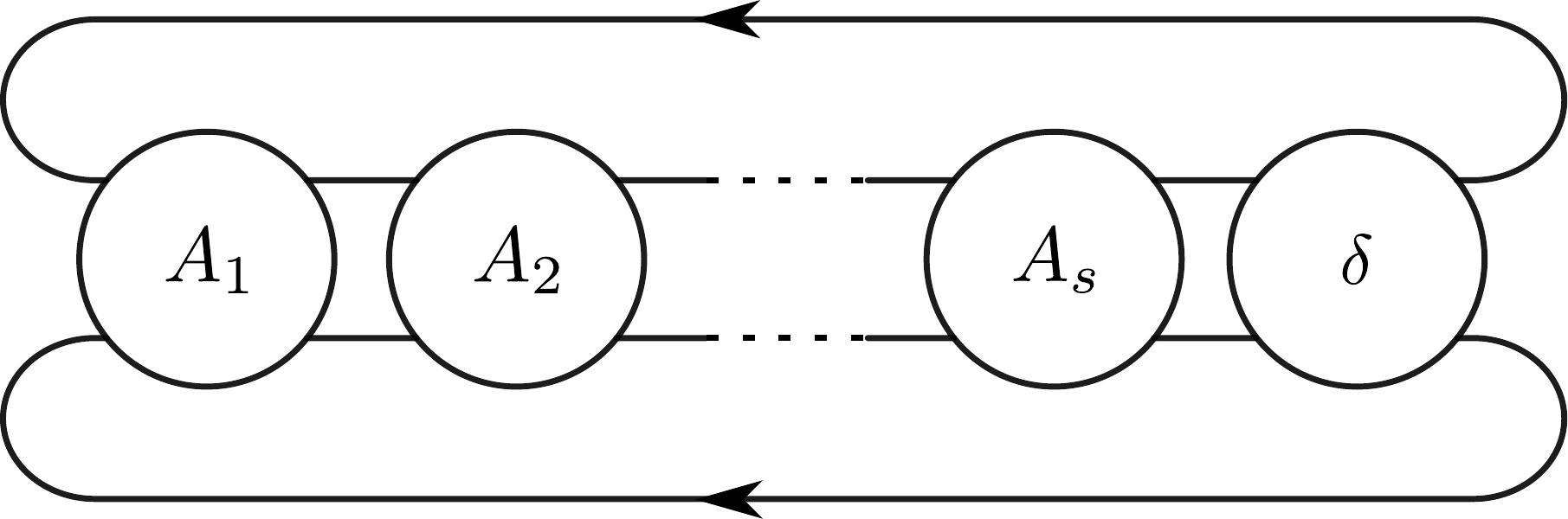}
\caption{A diagram depicting a general Montesinos knot with $s$ rational tangles and $\delta$ horizontal half-twists. The arrows indicate the potential orientation assignments if the knot is to be oriented.}
\label{Montesinos}
\end{figure}

The bridge index of a general Montesinos knot (alternating or nonalternating) is known and given by the following theorem.

\begin{theorem} \cite{BoZi} 
\label{bridgeMknot}
Let $K=M(\beta_1/\alpha_1,\ldots, \beta_s/\alpha_s,\delta)$ be a Montesinos knot with $|\beta_j/\alpha_j|<1$, $1\le j\le s$, then 
$\br(K)=s$.
\end{theorem} 

An explicit formula for the braid index of any alternating Montesinos knot is a rather new result \cite{DEHL2018}. The discussion from here to (and including) Theorem \ref{Montesinos_formula} is modified from \cite{DEHL2018}, as it is needed in order to understand the concepts and the formulas used in Theorem \ref{Montesinos_formula}. 

Let $\vec{K}=\vec{M}(\beta_1/\alpha_1,\ldots, \beta_s/\alpha_s,\delta)$ be an oriented Montesinos knot. Following \cite{DEHL2018}, we will use the following conventions on $\vec{K}$.
If $\vec{K}$ is alternating then all fractions $\beta_j/\alpha_j$ have the same sign and this is matched by the sign of $\delta$ representing the $|\delta|$ half-twists. As in the case of 2-bridge knots the sign of $\delta$ and the $\beta_j/\alpha_j$ should not be confused with the sign of individual crossings. For example, the signs of the crossings represented by $\delta$ may not coincide with the sign of $\delta$. The sign of the crossings represented by $\delta$ depends on the orientations of the two strings in the $\delta$-half twists. Since a knot and its mirror image have the same braid index, for alternating Montesinos knots we only need to consider the case  $\beta_j/\alpha_j>0$ for each $j$.

and that the crossings in the tangle diagrams are as chosen in the standard drawings of 2-bridge knots as shown in Figure \ref{tangle}. The conclusion we reach will then be applicable to the case $\beta_j/\alpha_j<0$ for each $j$ that mirrors the Montesinos knots discussed below. Furthermore, if $s=2$,  the a Montesinos knot is actually a 2-bridge knot and thus we only need to consider the case $s\ge 3$. 
For our purpose, we can always orient the top long strand in a Montesinos knot diagram from right to left as shown in Figure \ref{Montesinos} since reversing the orientations of all components in a knot does not change its braid index. 

We will use a standard drawing for each rational tangle $A_j$ which is given by the continued fraction of the rational number $\beta_j/\alpha_j$ and contains an odd number of positive entries, exactly like what we did in the case of 2-bridge knots. That is, we assume that $0<\beta_j<\alpha_j$ and $\beta_j/\alpha_j$ has a continued fraction decomposition of the form $(a_{1}^j,a^j_2,...,a_{2q_j+1}^j)$. The four strands that entering/exiting each tangle are marked as NW, NE, SW and SE. One example is shown at the left of Figure \ref{tangle}. Here we note that $a_{2q_j+1}^j$ is allowed to equal one if needed to ensure that the vector $(a_{1}^j,a^j_2,...,a_{2q_j+1}^j)$ has odd length. We have
$$
\frac{\beta_j}{\alpha_j}=\frac{1}{a_1^j+\frac{1}{a_2^j+\frac{1}{.....\frac{1}{a_{2q_j+1}^j}}}}. 
$$
The closure of a rational tangle is obtained by connecting its NW and SW end points by a strand and connecting its NE and SE end points with another strand (as shown at the left side of Figure \ref{tangle}). This closure is called the denominator $D(A_j)$ of the rational tangle $A_j$. Notice that $D(A_j)$ results in a normal standard diagram of  the oriented 2-bridge knot $\vec{B}(\alpha_j, \beta_j)$ given by the vector $(a_{1}^j,a^j_2,...,a_{2q_j+1}^j)$ (as shown at the right side of Figure \ref{tangle}). Finally, we define $(b_{1}^j,b^j_2,...,b_{2q_j+1}^j)$ by $|b^j_m|=a^j_m$ with its sign matching the signs of the corresponding crossings in $\vec{K}$ under the given orientation. Notice that $(b_{1}^j,b^j_2,...,b_{2q_j+1}^j)$ may be different from that defined for 2-bridge knots since the orientations of the strands here are inherited from $\vec{K}$. We will use the notation $A_j(b_{1}^j,b^j_2,...,b_{2q_j+1}^j)$ to denote the tangle $A_j$ and the signed vector associated with it that is inherited from the orientation of $\vec{K}$. Notice further that in a standard Montesinos knot diagram, each rational tangle has at the bottom a vertical row of $a_1$ twists corresponding to the condition $|\beta_j/\alpha_j|<1$. (Tangles with $|\beta_j/\alpha_j|\ge1$ end with a row of horizontal twists on the right, and these twists can be combined via flypes with the $\delta$ half twists to reduce the tangle to $|\beta_j/\alpha_j| \mod(1)$.)

\begin{figure}[htb!]
\includegraphics[]{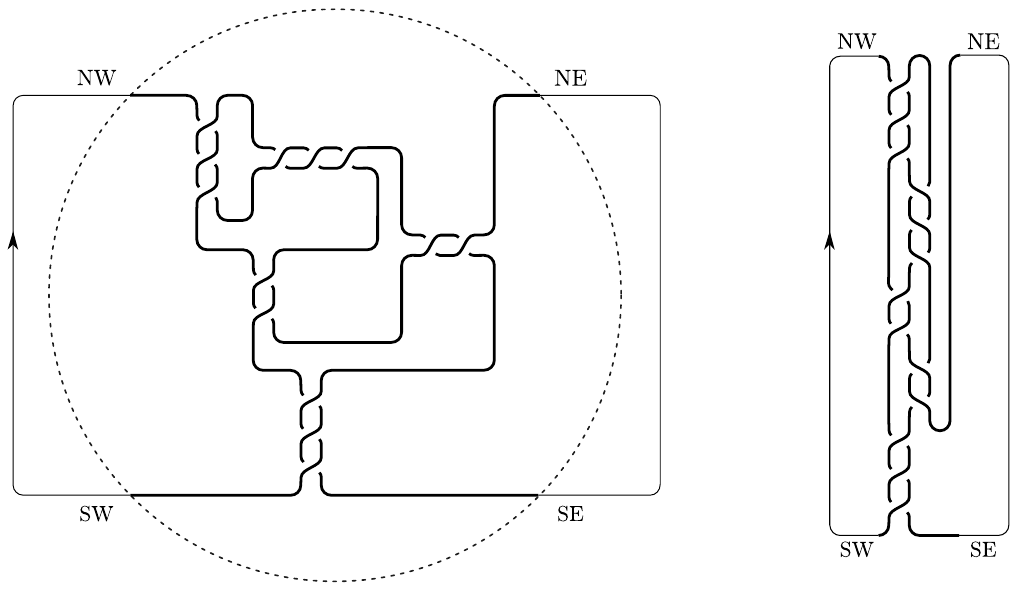}
\caption{Left: A standard drawing of the rational tangle $56/191 = A(3,2,2,3,3)$; Right: The denominator $D(A(3,2,2,3,3))$ is a standard 2-bridge knot diagram of the 2-bridge knot $K(191,56)$.}
\label{tangle}
\end{figure}

Let us now consider the Seifert circle decomposition of $\vec{K}$ by first examining how the arcs of Seifert circles entering and exiting each $A_j(b_{1}^j,b^j_2,...,b_{2q_j+1}^j)$ might look. Figure \ref{decomp} lists all eight possibilities for these arcs. Small Seifert circles within each tangle are not shown in Figure \ref{decomp}. Observing (from Figure \ref{tangle}) that the SW and SE strands meet at the last crossing in $b_{1}^j$, therefore if these two strands belong to two different Seifert circles, then they must have parallel orientation. Thus (vi) and (viii) are not possible. Furthermore, since we have assigned the top long arc in the Montesinos knot diagram the orientation from right to left, (iii) is not possible either. We say that $A_j$ is of {\em Seifert Parity 1} if it decomposes as (i) in Figure \ref{decomp}, of {\em Seifert Parity 2} if it decomposes as (ii) or (iv) in Figure \ref{decomp} and of {\em Seifert Parity 3} if it decomposes as (v) or (vii) in Figure \ref{decomp}. Note that the Seifert Parity of a tangle $A_j$ is not a property of the tangle alone, but depends on the structure of the diagram that contains $A_j$. (It should not be confused with the the term {\em parity of a tangle}, which denotes how the arcs in a tangle are connected.)

\begin{figure}[htb!]
\includegraphics{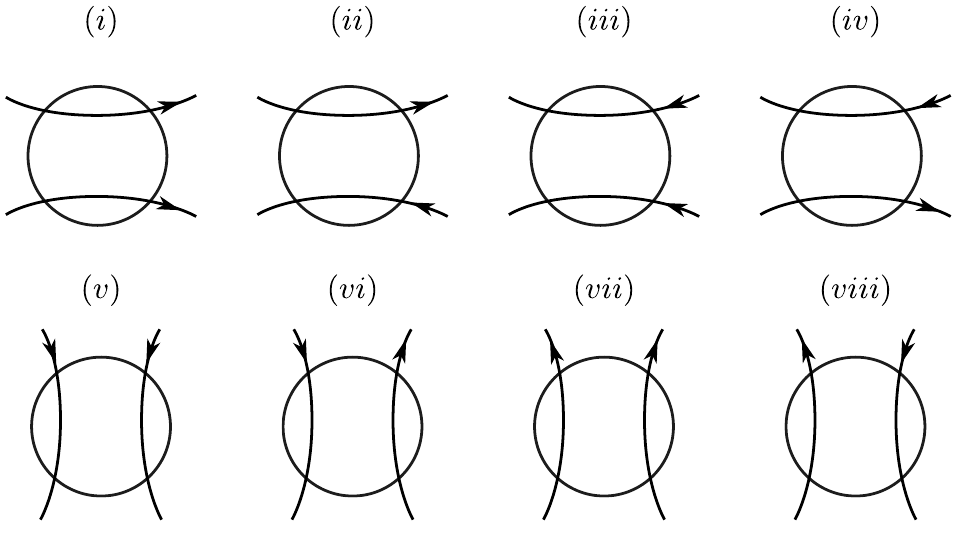}
\caption{Of the eight cases listed, (iii), (vi) and (viii) are not possible.}
\label{decomp}
\end{figure}

\medskip
\begin{remark}\label{bj_sign}{\em 
By assigning the appropriate orientations to the two strands at SW and SE in the bottom portion of Figure \ref{tangle}, one can relate the sign of $b_1^j$ to the Seifert Parity of $A_j$ as follows: $b_1^j<0$ if $A_j$ is of Seifert Parity 1 or 2 and  
$b_1^j>0$ if $A_j$ is of Seifert Parity 3.}
\end{remark}

Let us call a Seifert circle in the Seifert circle decomposition of $\vec{K}$ a {\em long} Seifert circle if it contains either the top long strand or/and the bottom long strand. $\vec{K}$ can be partitioned into the following three classes. 

Class M1. The top long strand and the bottom long strand belong to two different long Seifert circles, and the bottom long strand also has the orientation from right to left. Notice that $\vec{K}$ is of Class M1 if and only if every $A_j$ is of Seifert Parity 1 and all crossings in $\delta$ (if there are any) have negative signs.

Class M2. The top long strand and the bottom long strand belong to two different long Seifert circles and the bottom long strand has the orientation from left to right. Notice that $\vec{K}$ is of Class M2 if and only if every $A_j$ is of Seifert Parity 2 (more precisely case (ii) in Figure \ref{decomp}) and in this case $\delta=0$.

Class M3. The top long strand and the bottom long strand belong to the same long Seifert circle. Notice that $\vec{K}$ is of Class M3 if and only if all crossings in $\delta$ (if there are any) have positive signs and at least one tangle is of Seifert Parity 3 where we consider the horizontal half twists given by $\delta$ as an additional tangle besides the $A_j$ tangles.

\noindent
Let $i=1$, 2 or 3 be the Seifert Parity type of $A_j$ and define 

\begin{eqnarray}
\Delta_1(A_j)&=&\Delta_3(A_j)=\frac{(-1+\sign(b^j_{2q_j+1}))}{4}+\Delta(A_j), \label{Deltaoneorthree}\\
\Delta_2(A_j)&=&\frac{(2+\sign(b^j_1)+\sign(b^j_{2q_j+1}))}{4}+\Delta(A_j),\label{Deltatwo}
\end{eqnarray}
where 
$$
\Delta(A_j)=\sum_{b^j_{2m}>0,1\le m\le q_j}b^j_{2m}/2+\sum_{b^j_{2m+1}<0,0\le m\le q_j}|b^j_{2m+1}|/2.
$$

The braid index for an oriented and alternating Montesinos knot $\vec{K}$ is  given by the following theorem.

\begin{theorem}\label{Montesinos_formula} \cite{DEHL2018} 
Let $\vec{K}=\vec{M}(\beta_1/\alpha_1,\ldots, \beta_s/\alpha_s,\delta)=\vec{M}(A_1,A_2,\ldots, A_s,\delta)$ be an oriented and alternating Montesinos knot with a normal standard diagram and the signed vector $(b_{1}^j,b^j_2,...,b_{2q_j+1}^j)$ for $A_j$, then we have
\begin{eqnarray*}
\b(\vec{K})&=&2+\sum_{1\le j\le s}\Delta_1(A_j)\ {\rm{if}}\ \vec{K}\ {\rm{belongs\ to\ Class\ M1}};\\%
\b(\vec{K})&=&1+\sum_{1\le j\le s}\Delta_2(A_j) \ {\rm{if}}\ \vec{K}\ {\rm{belongs\ to\ Class\ M2}};\\%
\b(\vec{K})&=&\Delta_0(\vec{K})+\sum_{A_j \in \Omega_2}\Delta_2(A_j)+\sum_{A_j \in \Omega_3}\Delta_3(A_j) \ {\rm{if}}\ \vec{K}\ {\rm{belongs\ to\ Class\ M3}},%
\end{eqnarray*}
where $\Omega_2$, $\Omega_3$ are the sets of $A_j$'s that have Seifert Parity 2 and 3 respectively, $\Delta_0(\vec{K})=\eta+\delta-\min\{(\eta+\delta)/2-1,\delta\}$ and $\eta=\vert \Omega_3\vert$. 
\end{theorem}

Using Theorem \ref{Montesinos_formula} we can now identify the alternating Montesinos knots where the bridge index equals the braid index.

\begin{theorem}
\label{braidequalsbridgeMknot}
Let $ \vec{K}=\vec{M}(\beta_1/\alpha_1,\ldots, \beta_s/\alpha_s,\delta)$ be an oriented and alternating Montesinos knot. Then $\b( \vec{K})=\br( K)=s$ if and only if the following conditions hold:

(i) $\vec{K}$ is of Class M3.

(ii) $\eta\ge \delta+2$ where $\eta=\vert \Omega_3\vert$ is the number of tangles $A_j$ with Seifert Parity 3.

(iii) If $A_j$  has Seifert Parity 2 then $A_j=(-1,-s_j,-1)$ for some $s_j\ge 0$. If $A_j$  has Seifert Parity 3 then $A_j=(s_j)$ for some $s_j> 0$.
\end{theorem}

\begin{proof}
``$\Longrightarrow$": Assume that $\vec{K}=\vec{M}(\beta_1/\alpha_1,\ldots, \beta_s/\alpha_s,\delta)$ is an oriented alternating Montesinos knot with $\br(K)=\b(\vec{K})=s$.

If $\vec{K}$ is of Class M1 then every $A_j$ has Seifert Parity 1 and there are at least two $A_j$'s with $\Delta_1(A_j)=0$ since $\Delta_1(A_j)\ge 0$ for each $j$.
By Remark \ref{bj_sign}, $\sign(b^j_1)=-1$. Thus 
$\Delta_1(A_j)=0$ is possible only if  $\sign(b^j_{2q_j+1})=-1$ in (\ref{Deltaoneorthree}).
However then $\Delta_1(A_j)\ge \sum_{b^j_{2m+1}<0,0\le m\le q_j}|b^j_{2m+1}|/2>0$ and therefore $\Delta_1(A_j)=0$ is not possible.

Similarly if $\vec{K}$ is of Class M2 then there exists at least one tangle $A_j$ with $\Delta_2(A_j)=0$. This is only possible if $\sign(b^j_1)=\sign(b^j_{2q_j+1})=-1$ in (\ref{Deltatwo}).
However in such a case \newline
$\Delta_2(A_j)\ge \sum_{b^j_{2m+1}<0,0\le m\le q_j}|b^j_{2m+1}|/2>0$ therefore $\Delta_2(A_j)=0$ is not possible either. This proves that $\vec{K}$ must of Class M3 and therefore all tangles must have Seifert Parity 2 or 3.

Let $\eta$ be the number of tangles $A_j$ with Seifert Parity 3 and $s-\eta$ be the number of tangles $A_j$ with Seifert Parity 2. We note that the contribution $\Delta_2(A_j)$ of a tangle with Seifert Parity 2 to the braid index is always an integer (it is just one less than the braid index of the corresponding two bridge link, see Theorem \ref{2bridgeformula}) and therefore we know that a tangle of Seifert parity 2 must have $\Delta_2(A_j)\ge1$. We have
$$
s=\b(\vec{K})\ge \Delta_0(\vec{K})+s-\eta +\sum_{A_j \in \Omega_3}\Delta_3(A_j)  \ge s+\delta-\min\{(\eta+\delta)/2-1,\delta\}.
$$
This implies that $0=\delta-\min\{(\eta+\delta)/2-1,\delta\}$, hence $\eta\ge \delta+2$.
Moreover we must have $\Delta_2(A_j)=1$ for each $A_j$ with Seifert Parity 2 and $\Delta_3(A_j)=0$ for each $A_j$ with Seifert Parity 3.

If $A_j$ has Seifert Parity 2 and $\Delta_2(A_j)=1$, then $\sign(b^j_1)=-1$ by Remark \ref{bj_sign} and equation (\ref{Deltatwo}) becomes
$$
1=\Delta_2(A_j)=\frac{|b^j_1|}{2}+\frac{(1+\sign(b^j_{2q_j+1}))}{4}+\sum_{b^j_{2m}>0,1\le m\le q_j}b^j_{2m}/2+\sum_{b^j_{2m+1}<0,0< m\le q_j}|b^j_{2m+1}|/2\ge \frac{|b^j_1|}{2}.
$$
Thus $|b^j_1|\le 2$. If $b^j_1=-2$ then we must have $\sign(b^j_{2q_j+1})=-1$, however $1=\Delta_2(A_j)$ is only possible if $q_j=0$ and $A_j=(-2)$, which can be written as $(-1,0,-1)$ since $(2)$ and $(1,0,1)$ are both continued fraction decompositions of $1/2$. If $b^j_1=-1$ then $2q_j+1\ge 3$ and we must have $\sign(b^j_{2q_j+1})<0$ hence $b^j_{2q_j+1}=-1$ as well. In addition we can see that both $\sign(b^j_2)$ and $\sign(b^j_3)$ are negative. Thus we must have  $2q_j+1= 3$, that is, $A_j=(-1,-s_j,-1)$ for some integer $s_j>0$. 

If $A_j$ has Seifert Parity 3, then $\Delta_3(A_j)=0$. $b^j_1>0$ by Remark \ref{bj_sign}. If $2q_j+1\ge 3$ then one can easily check that $\sign(b^j_2)>0$, which would lead to $\Delta_3(A_j)>0$. Thus $2q_j+1=1$ and $A_j=(s_j)$ for some integer $s_j>0$. 
Thus we have shown that conditions (i), (ii) and (iii) are satisfied.

``$\Longleftarrow$": This is straight forward by Theorem \ref{Montesinos_formula}.
 \end{proof}

\begin{definition}\label{v_h}{\em 
Let us call a tangle $A_j$ corresponding to a rational number of the form $\beta/\alpha=1/s_j$ for some $s_j\ge 2$ a {\em vertical} tangle, and a tangle $A_j$ corresponding to a rational number of the form $\beta/\alpha=(s_j+1)/(s_j+2)$ for some $s_j\ge 1$ a {\em horizontal} tangle. 
}
\end{definition}

Note that the standard continued fraction decompositions of a vertical and a horizontal tangle are of the form $(s_j)$ and $(1,s_j,1)$ respectively. Using Theorem \ref{braidequalsbridgeMknot}, we can then identify all BB knots in the set of all un-oriented alternating Montesinos knots as stated in the following theorem.
 
\begin{theorem}
\label{unorientedM_BBknots}
Let ${K}={M}(\beta_1/\alpha_1,\ldots, \beta_s/\alpha_s,\delta)$ be an un-oriented alternating Montesinos knot. Then $\b({K})=\br( K)=s$ if and only if $K$ or its mirror image satisfies the following conditions:

(i) every $A_j$ is either a vertical tangle or a horizontal tangle;

(ii) $\eta\ge \delta+2$ where $\eta$ is the number of $A_j$'s that are vertical.
\end{theorem}

The key observation one needs to make in the proof of Theorem \ref{unorientedM_BBknots} is that under condition (ii), we can oriented $K$ so that the resulting $\vec{K}$ belongs to Class M3. We leave the details to the reader.
 
The BB alternating Montesinos knots with one component and crossing number up to 12 are listed in Table \ref{table} indicated by a superscript $^\dagger$.

\subsection{Non alternating Montesinos knots}

A classification of Montesinos knots (including both the alternating and non-alternating cases) can be found in \cite{BZ}. A non alternating Montesinos knot will have a minimum diagram when there are no integral twists on the right, that is, $\delta=0$. We have the following:

\begin{theorem}
\label{braidequalsbridgeMknotnonalt}
Let $\vec{K}=\vec{M}(\beta_1/\alpha_1,\ldots, \beta_s/\alpha_s,0)$ be an oriented Montesinos knot. If $\vec{K}$ belongs to Class M3, $A_j=(-1,-s_j,-1)$ for some integer $s_j\ge 0$ if it has Seifert Parity 2 and $A_j=(s_j)$ for some integer $s_j\ge 2$ if it has Seifert Parity 3, then $\br(K)=\b(\vec{K})$.
\end{theorem}

\begin{proof} Let $\vec{K}$ be an oriented Montesinos knot satisfying conditions (i) and (ii). 
It is known that $ \b(\vec{K})\le \gamma(\vec{K})$ \cite{Ya}, where $\gamma(\vec{K})$ denotes the number of Seifert circles in the Seifert circle decomposition of $\vec{K}$. Thus we have $\br(K)\le \b(\vec{K})\le \gamma(\vec{K})$, and the result of the theorem follows if we can show that $\gamma(\vec{K})=s$ since $s=\br(K)$ by Theorem \ref{bridgeMknot}.
Each tangle $A_j$ with Seifert Parity 2 contributes a Seifert circle (which is contained within the tangle $A_j$).
 Since each Seifert circle that is not completely contained within a tangle (this includes the large Seifert circle) must contain two vertical arcs: one each from a tangle of case (v) and from a tangle of case (vii) as defined in Figure \ref{decomp} hence each tangle $A_j$ of Seifert Parity 3 also contributes one Seifert circle. It follows that $\vec{K}$ has exactly $s$ Seifert circles.
\end{proof}

\begin{corollary}\label{non_un_cor}
Let ${K}={M}(\beta_1/\alpha_1,\ldots, \beta_s/\alpha_s,0)$ be a Montesinos knot with $s\ge 3$ (that is not necessarily alternating). If each $A_j$ or its mirror image is either a vertical or a horizontal tangle (as defined in Definition \ref{v_h}) and at least two of them are vertical, then $\br(K)=\b({K})$.
\end{corollary}

If $K$ satisfies the conditions in the corollary, then we can oriented it in a way that it satisfies the conditions of Theorem \ref{braidequalsbridgeMknotnonalt}. The details are left to the reader.

\begin{remark}{\em
We note that the proof of Theorem \ref{braidequalsbridgeMknotnonalt} does not depend on whether $\vec{K}$ is alternating, thus it also works for an oriented alternating Montesinos knot $\vec{K}=\vec{M}(\beta_1/\alpha_1,\ldots, \beta_s/\alpha_s,\delta)$ where $\delta=0$. However this method is not powerful enough to prove Theorem \ref{braidequalsbridgeMknot} since the Seifert circle decomposition of $\vec{K}$ contains $s+\delta$ Seifert circles when $\vec{K}$ satisfies the conditions of Theorem \ref{braidequalsbridgeMknot}. That is, $\vec{K}$ (in its standard diagram form) does not minimize the number of Seifert circles. In \cite{DEHL2018} we explain an algorithm that converts $\vec{K}$ to a non alternating and non minimal diagram $D$ that minimizes the number of Seifert circles via so called reduction moves.}
\end{remark}

All non alternating Montesinos knots with one component and up to 12 crossings that satisfy the condition of 
Corollary \ref{non_un_cor} are listed in Table \ref{table} and are marked with a superscript $^\ddagger$. It turns out that this is a complete list of one component BB non alternating Montesinos knots with crossing number up to 12. In fact, we conjecture that this is generally true. See Conjecture \ref{conjecture1} at the end of the paper.

\subsection{Knots using Conway basic polyhedra}

In \cite{Con67} the concept of Conway basic polyhedra was introduced. Here we will only concentrate on one class of such polyhedra.
The Conway basic polyhedra $n^*$ for $n=2k$ and $k\ge 3$ can be thought of as taking a regular $k$-gon into which we inscribe a second regular $k$-gon such that its vertices touch the midpoints of the sides of the original $k$-gon. Into the smaller $k$-gon we inscribe a third $k$-gon in the same way. The left of Figure \ref{eightstar} illustrate the case of $k=4$. This will form a knot  diagram $D$ with $n$ crossings if we make the diagram alternating. Moreover the writhe $w(D)=0$ (with proper orientation in case of a link), three Seifert circles and braid index 3. For example $8^*$ is the knot $8_{18}$ (the right of Figure \ref{eightstar}) and $10^*$ is the knot $10_{123}$. 
We note that for $n\not\equiv 0\mod(3)$ we obtain a knot diagram and for $n\equiv 0\mod(3)$ we obtain a 3-component knot diagram, for example  for $k=3$ we obtain the Borromean rings ($6_2^3$ in Rolfsen notation). Note that when he first introduced these concepts, Conway used two symbols for the six crossing diagram ($k=3$) denoted by $6^{*}$ and $6^{**}$.  The two symbols denote isomorphic graphs that were introduced to make the obtained notation for knots to be nicer. However they are essentially the same basic polyhedron with alternative way to insert tangles. Conway developed a shorthand notation where the two symbols $6^{**}$ and $6^{*}$ are omitted. If there is an initial dot in the beginning of the symbol, then it is $6^{**}$. If there are more than one tangle symbols separated by dots without the basic polyhedron symbol, then it is meant to be $6^{**}$. If there is also a dot in the beginning of the symbol, then it is $6^{*}$.
If the tangle substitutions are simple then there is a natural way to view these diagrams as a 3 string braids without increasing the crossing number and we obtain a braid of three strings and bridge number three.

\begin{figure}[htb!]
\includegraphics[scale=.15]{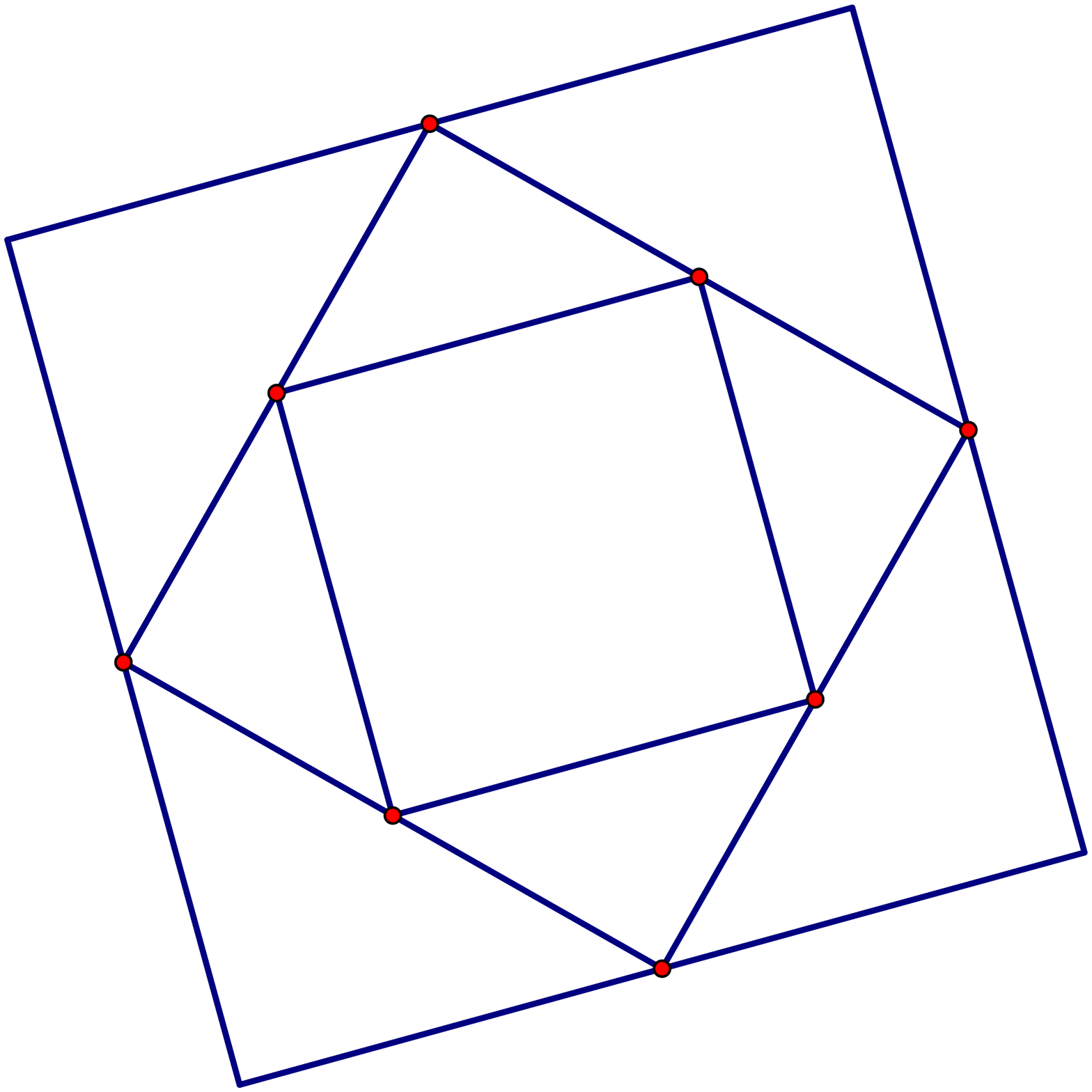}\qquad
\includegraphics[scale=.15]{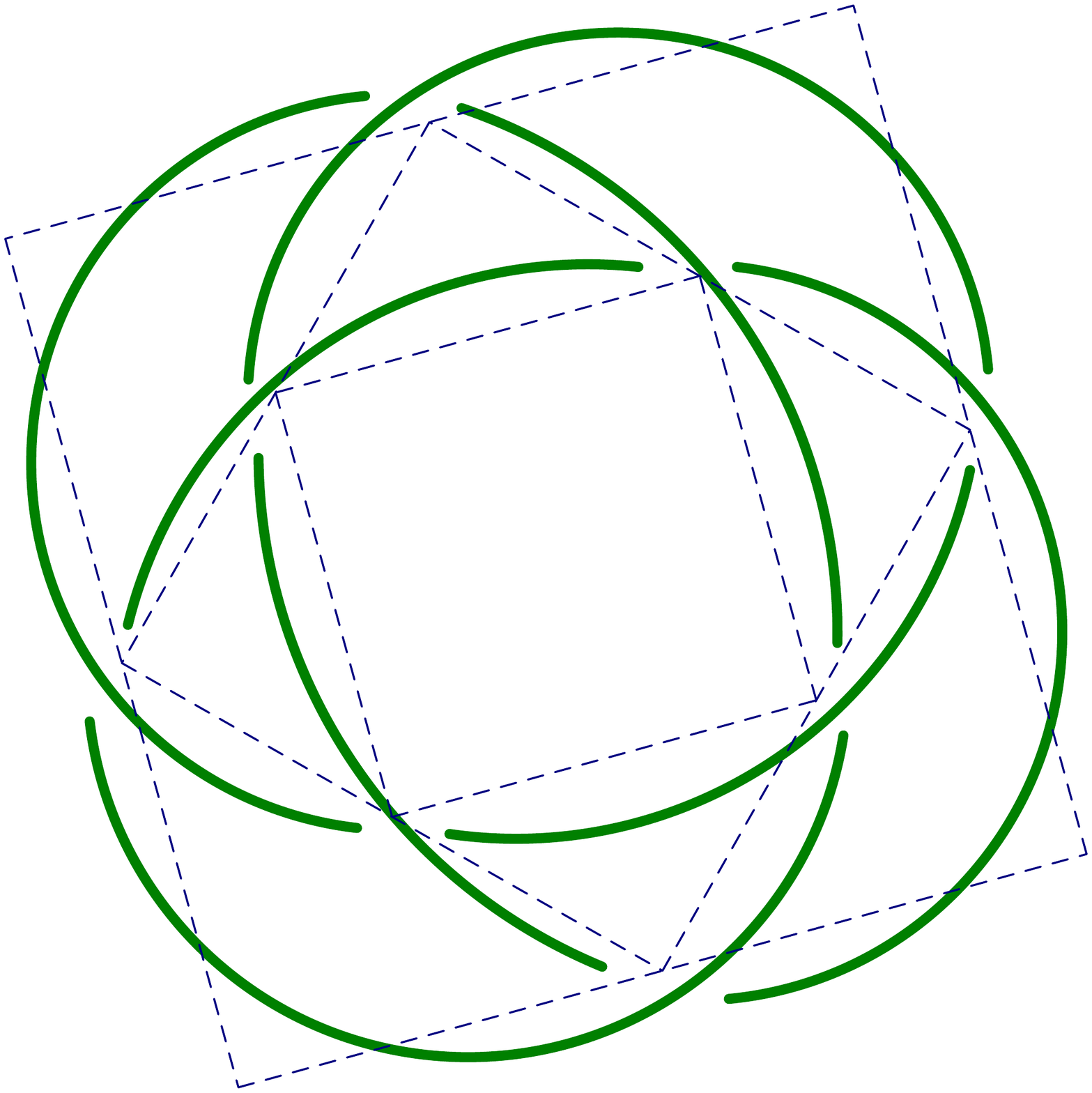}
\caption{Left: The Conway basic polyhedra $8^*$. Right: The knot $8_{18}$. }
\label{eightstar}
\end{figure}

Using the $6^{*}$, $6^{**}$, $8^{*}$ and $10^{*}$ diagrams, we are able to identify all one component BB knots with crossing number up  to 12 that can be constructed using Conway basic polyhedra. They are listed in 
Table \ref{table} and are marked with a superscript $^*$. We would like to point out that 
while the above knot construction using Conway basic polyhedra can be generalized, the generalized construction may not yield a BB knot.
One such example is $n^*$ with $n = 3k$ that can be thought of as inscribing four $k$-gons into each other. The simplest example is $9^*$ which forms the knot $9_{40}$. It has a diagram with 4 Seifert circles and its braid index is four, but its bridge number is only three.

\subsection{ Arborescent knots or Conway algebraic knots.}

Conway algebraic knots are a super family of the Montesinos knots \cite{BoSi, Con67}.
The following describes Conway algebraic knots with bridge and braid index three.
First, they are represented in the Conway notation as $[(a_1;b_1)(a_2;b_2)]$ with $|a_i|\ge 2$ and $|b_i|\ge 2$ \cite{Con67}. These Conway algebraic knots may have one or more components, depending on the parity of $a_i$ and $b_i$. For example, if $a_1$, $a_2$ are both even and $b_1$, $b_2$ are both odd then the knot  has only one component. In the case of multiple components we need to assume that the components are  oriented such that $a_i$ and $b_i$ represent parallel oriented half twists. These knots can be either alternating or non-alternating. In the non alternating case they can also be denoted by $[(a_1;b_1)-(a_2;b_2)]$, $[(a_1;b_1)(a_2-1,1;-b_2)]$, $[(a_1-1,1;-b_1+1,-1)(a_2;b_2)]$ or  $[(-a_1+1,-1;b_1-1,1)(a_2;b_2)]$. The braid index of such a Conway algebraic knot is at most $3$ since the standard diagram given by the Conway notation has 3 Seifert circles in its Seifert circle decomposition. Since these are not 2-bridge knots, their bridge index is at least 3. Thus if $K$ is a Conway algebraic knot described here, then we have $\b(K)=\br(K)=3$.  

\subsection{Summary of BB knots with crossing number up to 12.}

Table \ref{table} lists all one component BB knots with crossing number up to 12. The superscriptions are used in the table as follows. $^t$: torus knots; $^\dagger$: alternating Montesinos knots; $^\ddagger$: non alternating Montesinos knots; $^*$: knots constructed from Conway basic polyhedra; $^\flat$: Conway algebraic knots. The symbols used for the knots follow the historical notation for knots with at most 10 crossings \cite{rolfsen2003knots} and the notation created for the program Knotscape, developed by Morwen and Jim Hoste \cite{hoste1999knotscape} for knots with more than 10 crossings, where the letters a and n denote alternating and non-alternating knots, respectively. We would like to point out that a few BB knots in Table \ref{table} belonging to several families discussed above. For example the knot $8_{19}$ is the torus knot $T(4,3)$, the non alternating Montesinos knot corresponding to $[3;3;2;-]$, as well as the knot constructed using the $8^*$ Conway basic polyhedron. Another example is the knot $10_{124}$: it is the torus knot $T(5,3)$ and also the non alternating Montesinos knot $[5;3;2;-]$.

\begin{center}
\begin{longtable}{|llllllll|}
\hline\hline
 $(3_1)^t$&&&&&&&\\ \hline
 $(5_1)^t$&&&&&&&\\ \hline
 $(7_1)^t$&&&&&&&\\ \hline
 $(8_{5})^\dagger$&$(8_{10})^\dagger$&$(8_{16})^\dagger$&$(8_{17})^*$&$(8_{18})^*$&
 $(8_{19})^{t*\ddagger}$&$(8_{20})^\ddagger$&$(8_{21})^\ddagger$\\ \hline
 $(9_1)^t$&$(9_{16})^\dagger$&&&&&&\\ \hline
 $(10_{46})^\dagger$&$(10_{47})^\dagger$&
 $(10_{48})^\dagger$&$(10_{62})^\dagger$&$(10_{64})^\dagger$&$(10_{79})^\flat$&
$(10_{82})^*$&$(10_{85})^*$\\ \hline
$(10_{91})^*$&$(10_{94})^*$&
$(10_{99})^*$&$(10_{100})^*$&$(10_{104})^*$&$(10_{106})^*$&
$(10_{109})^*$&$(10_{112})^*$\\ \hline
$(10_{116})^*$&$(10_{118})^{t*}$&$(10_{123})^*$&$(10_{124})^{t\ddagger}$&
$(10_{125})^\ddagger$&$(10_{126})^\ddagger$&$(10_{127})^\ddagger$&$(10_{139})^\ddagger$\\ \hline
$(10_{141})^\ddagger$&$(10_{143})^\ddagger$&$(10_{148})^\flat$&$(10_{149})^\flat$&
$(10_{152})^\flat$&$(10_{155})^*$&$(10_{157})^*$&$(10_{159})^*$\\ \hline
$(10_{161})^*$&&&&&&&\\ \hline
$(11a_{44})^\dagger$&$(11a_{47})^\dagger$&$(11a_{57})^\dagger$&$(11a_{231})^\dagger$&
$(11a_{240})^\dagger$&$(11a_{263})^\dagger$&$(11a_{338})^\dagger$&$(11a_{367})^t$\\ \hline
$(11n_{71})^\ddagger$&$(11n_{72})^\ddagger$&$(11n_{73})^\ddagger$&$(11n_{74})^\ddagger$&
$(11n_{75})^\ddagger$&$(11n_{76})^\ddagger$&$(11n_{77})^\ddagger$&$(11n_{78})^\ddagger$\\ \hline
$(11n_{81})^\ddagger$&&&&&&&\\ \hline
$(12a_{146})^\dagger$&$(12a_{167})^\dagger$&$(12a_{369})^\dagger$&$(12a_{576})^\dagger$&
$(12a_{692})^\dagger$&$(12a_{801})^\dagger$&$(12a_{805})^*$&
$(12a_{815})^\flat$\\ \hline
$(12a_{819})^*$&$(12a_{824})^\flat$&$(12a_{835})^\dagger$&
$(12a_{838})^\dagger$&$(12a_{850})^*$&$(12a_{859})^*$&$(12a_{864})^*$&
$(12a_{869})^*$\\ \hline
$(12a_{878})^\dagger$&$(12a_{898})^*$&$(12a_{909})^*$&
$(12a_{916})^*$&$(12a_{920})^*$&$(12a_{981})^*$&$(12a_{984})^*$&
$(12a_{999})^*$\\ \hline
$(12a_{1002})^*$&$(12a_{1011})^*$&$(12a_{1013})^*$&
$(12a_{1027})^\dagger$&$(12a_{1047})^*$&$(12a_{1051})^*$&$(12a_{1114})^*$&
$(12a_{1120})^*$\\ \hline
$(12a_{1168})^*$&$(12a_{1176})^*$&$(12a_{1191})^*$&
$(12a_{1199})^*$&$(12a_{1203})^*$&$(12a_{1209})^*$&$(12a_{1210})^*$&
$(12a_{1211})^*$\\ \hline
$(12a_{1212})^*$&$(12a_{1214})^\dagger$&$(12a_{1215})^*$&
$(12a_{1218})^*$&$(12a_{1219})^*$&$(12a_{1220})^*$&$(12a_{1221})^*$&
$(12a_{1222})^*$\\ \hline
$(12a_{1223})^*$&$(12a_{1225})^*$&$(12a_{1226})^*$&
$(12a_{1227})^*$&$(12a_{1229})^*$&$(12a_{1230})^*$&$(12a_{1231})^*$&
$(12a_{1233})^\dagger$\\ \hline
$(12a_{1235})^*$&$(12a_{1238})^*$&$(12a_{1246})^*$&
$(12a_{1248})^*$&$(12a_{1249})^*$&$(12a_{1250})^*$&$(12a_{1253})^*$&
$(12a_{1254})^*$\\ \hline
$(12a_{1255})^*$&$(12a_{1258})^*$&$(12a_{1260})^*$&
$(12a_{1283})^\dagger$&$(12a_{1288})^\flat$&$(12n_{113})^\flat$&$(12n_{114})^\flat$&
$(12n_{190})^\flat$\\ \hline
$(12n_{191})^\flat$&$(12n_{233})^\ddagger$&$(12n_{234})^\ddagger$&
$(12n_{235})^\ddagger$&$(12n_{242})^\ddagger$&$(12n_{344})^\flat$&$(12n_{345})^\flat$&
$(12n_{417})^*$\\ \hline
$(12n_{466})^\ddagger$&$(12n_{467})^\ddagger$&$(12n_{468})^\ddagger$&
$(12n_{472})^\ddagger$&$(12n_{570})^\ddagger$&$(12n_{571})^\ddagger$&$(12n_{574})^\ddagger$&
$(12n_{604})^\flat$\\ \hline
$(12n_{640})^*$&$(12n_{647})^*$&$(12n_{666})^*$&
$(12n_{674})^\flat$&$(12n_{675})^\flat$&$(12n_{679})^\flat$&$(12n_{683})^*$&
$(12n_{684})^*$\\ \hline
$(12n_{688})^\flat$&$(12n_{707})^*$&$(12n_{708})^*$&
$(12n_{709})^*$&$(12n_{721})^\ddagger$&$(12n_{722})^\ddagger$&$(12n_{725})^\ddagger$&
$(12n_{747})^*$\\ \hline
$(12n_{748})^*$&$(12n_{749})^*$&$(12n_{750})^*$&
$(12n_{751})^*$&$(12n_{767})^*$&$(12n_{820})^*$&$(12n_{821})^*$&
$(12n_{822})^*$\\ \hline
$(12n_{829})^*$&$(12n_{830})^*$&$(12n_{831})^*$&
$(12n_{850})^*$&$(12n_{882})^*$&$(12n_{887})^*$&$(12n_{888})^\flat$&       
 \\\hline\hline
 \caption{List of one component BB knots with crossing number up to 12. \label{table}}
\end{longtable}
\end{center}

\vspace{-0.7in}
\section{The number of BB knots with a given crossing number}
\label{numberofknots}

Table \ref{table2} summarizes the number of one component BB knots with crossing numbers up to 12. By separating the knots with odd crossing numbers and the knots with with even crossing number, we observe that the number of one component BB knots with a given crossing number $n$ increases as the crossing number increases and it seems that there are more BB knots in the case of an even crossing number greater than six; and (ii) the percentage of BB knots among all knots with the same crossing number decreases as the crossing number increases. What are the general behaviors of these numbers and percentages? More specifically, if we let $\mathcal{K}_n$ and $\mathcal{B}_n$ be the numbers of knots and BB knots with crossing number $n$ respectively, then how do $\mathcal{B}_n$ and $\mathcal{B}_n/\mathcal{K}_n$ behave? While it is 
plausible that $\lim_{n\to\infty}\mathcal{B}_n/\mathcal{K}_n=0$, we do not have a proof of it. However, we can prove that $\lim_{n\to\infty}\mathcal{B}_n=\infty$ and does so exponentially. This is stated in Theorem \ref{exp_thm}.

\begin{table}[!hb]
\begin{center}
\begin{tabular}{|c|c|c|c|c|c|c|c|c|c|c|c|c|c|} 
\hline  crossing \# &3&4&5&6&7&8&9&10&11&12\\
\hline \# of knots &1&1&2&3& 7&21&49&165&552&2176\\
\hline \# of BB knots &1&0&1&0&1&8&2&33&17&119\\
\hline $\approx$ percentage &100\%&0\%&50\%&0\%& 14.3\%&38.1\%&4.1\%&20.0\%&3.1\%&5.5\%\\
\hline
\end{tabular}
\end{center}
 \caption{Number of one component BB knots with respect to the crossing number. \label{table2}}
\end{table}

\begin{theorem}\label{exp_thm}
The number $\mathcal{B}_n$ grows exponentially with $n$. Moreover the number of one component BB knots with crossing number $n$ also grows exponentially with $n$.
\end{theorem}
\begin{proof}
In fact, we will prove that the number of BB knots with braid index 3 grows exponentially with the crossing number $n$. Assume that we only generate alternating 3-braids. Then a word representing such a braid can be generated by two symbols $\sigma_1$ and $\sigma_2^{-1}$. Using these two symbols there are $2^n$ words of length $n$ each of which representing an alternating braid with $n$ crossings.
If we close such a braid we obtain an alternating knot or link diagram with $n$ crossings and at most 3 components.
Now we need to estimate how many of these diagrams represent distinct knots of $n$ crossings. There are several points we need to consider:

(i) If $w$ is a braid word then any cyclic permutation of this word will yield the same link. Thus up to cyclic permutation there are at least $\frac{2^n}{n}$ different words.

(ii) Switching between $\sigma_1$ and $\sigma_2^{-1}$ in a braid word will result in the same knot. Thus up to cyclic permutation and exchange of $\sigma_1$ and $\sigma_2^{-1}$ there are $\frac{2^{n-1}}{n}$ different words (at least).

(iii) The diagram generated by the braid closure is a reduced alternating knot diagram if $\sigma_1$ and $\sigma_2^{-1}$ both occur in the word more than once.

(iv) Using the classification theorem in \cite{Bir93}, we see that a knot diagram obtained by closing a braid satisfying condition (iii) above admits flypes only if the braid is of the form $\sigma_1^u \sigma_2^{-1}\sigma_1^z\sigma_2^{-v}$, $\sigma_2^{-u}\sigma_1\sigma_2^{-z}\sigma_1^{v}$, $\sigma_1^u\sigma_2^{-v}\sigma_1^z\sigma_2^{-1}$ or  $\sigma_2^{-u}\sigma_1^v\sigma_2^{-z}\sigma_1^{1}$  for some positive integers $u$, $z$, and $v$. We note that $\sigma_1^u \sigma_2^{-1}\sigma_1^z\sigma_2^{-v}$ is flype equivalent to $\sigma_1^u\sigma_2^{-v}\sigma_1^z \sigma_2^{-1}$ and $\sigma_2^{-u}\sigma_1\sigma_2^{-z}\sigma_1^{v}$ is flype equivalent to $\sigma_2^{-u}\sigma_1^v\sigma_2^{-z}\sigma_1^{1}$.

It is easy to see that the number of braids satisfying condition (iii) and also admitting flypes is bounded above by $4\cdot {{n-2}\choose{2}}<n^3$. Furthermore, the number of words up to cyclic permutation and exchange of $\sigma_1$ and $\sigma_2^{-1}$ with at least two each of $\sigma_1$ and $\sigma_2^{-1}$ in the word is at least $2^{n-5}/n$. The alternating knots obtained by closing braids satisfying condition (iii) and that do not admit flypes are all distinct, thus there exists a constant $c>0$ such that when $n$ is large enough the number of such knots is bounded below by $2^{n-5}/n-n^3> 2^{cn}$. 
Next, we need to estimate how many of these 3 braids are 2-bridge knots using Theorem \ref{2bridge_theorem}. Using a case by case analysis similar to the method used in the proof of Theorem \ref{torustwobridge} we can list all vectors of 2-bridge knots with braid index three. If we represent the these 2-bridge knots with an odd length vector using positive integers than we can show that these vector are of the following form:
$211$, $a12$, $a2b$ and $3a1$  for length three vectors and $a11b1$ and $1a3b1$ for length five vectors,
where $a,b$ are positive integers. No vectors of a length greater than five can have a braid index of three. Thus for a fixed crossing number $n$ the number of different 2-bridge knots with braid index three is linearly bounded above by $c_1 n$  some fixed constant $c_1>0$. So, the number of  3-braids with $\br(K)=\b(K)=3$ will still growth exponentially. Roughly a third of these will be one component knots hence the number of one component BB knots also grows exponentially with $n$.
\end{proof}

We end our paper with the following two questions and a conjecture.

\begin{question} {\em %
Is $\lim_{n\rightarrow\infty}\mathcal{B}_n/{{\mathcal{K}_n}}=0$ and does this limit approach zero exponentially fast?
}
\end{question}

\begin{question} {\em Is it true that $\mathcal{B}_{2n}>\mathcal{B}_{2n+1}$ for all $n\ge 4$? And why?}
\end{question}

\begin{conjecture}\label{conjecture1}{\em 
The non alternating Montesinos knots described in Theorem \ref{braidequalsbridgeMknotnonalt} (or Corollary \ref{non_un_cor}) are the only non alternating Montesinos knots that are BB knots. 
}
\end{conjecture}

\subsection*{Acknowledgement}
Ph.~R.\@ was supported by German Research Foundation
grant RE~3930/1--1. We would like to thank S\"oren Bartels for fruitful
discussions on the numerical simulation of elastic knots.

\newcommand{\href}[2]{#2}


\begin{thebibliography}{10}

\bibitem{sossinsky}
S.~Avvakumov and A.~Sossinsky.
\newblock \href {http://dx.doi.org/10.1134/S1061920814040013} {On the normal
  form of knots}.
\newblock {\em Russ. J. Math. Phys.}, \textbf{21}(4): 421--429, 2014.

\bibitem{bartels13}
S.~Bartels.
\newblock \href {https://doi.org/10.1093/imanum/drs041} {A simple scheme for
  the approximation of the elastic flow of inextensible curves}.
\newblock {\em IMA J. Numer. Anal.}, \textbf{33}(4):1115--1125, 2013.

\bibitem{knotevolve}
S.~Bartels, {\relax Ph}.~Falk, and P.~Weyer.
\newblock \href {https://aam.uni-freiburg.de/agba/forschung/knotevolve/}
  {{KNOT}evolve -- a tool for relaxing knots and inextensible curves}.
\newblock Web application, \url{https://aam.uni-freiburg.de/agba/forschung/knotevolve/}, 2020.

\bibitem{BR}
S.~{Bartels} and {\relax Ph}.~{Reiter}.
\newblock \href {http://dx.doi.org/10.1090/mcom/3633}{Stability of a simple
  scheme for the approximation of elastic knots and self-avoiding inextensible
  curves}.
\newblock {\em Math. Comp.} \textbf{90}:1499-1526, 2021.

\bibitem{BRR}
S.~Bartels, {\relax Ph}.~Reiter, and J.~Riege.
\newblock \href {http://dx.doi.org/10.1093/imanum/drx021} {A simple scheme for
  the approximation of self-avoiding inextensible curves}.
\newblock {\em IMA Journal of Numerical Analysis}, \textbf{38}(2): 543--565, 2018.

\bibitem{Bir93} J.~Birman and W.~Menasco
{\em Studying Links via Closed Braids, III. Classifying Links Which Are Closed 3-braids}, Pacific J. Math., \textbf{161}: 25--113, 1993.


\bibitem{BoSi} F.~Bonahon and L.~Siebermann. New Geometric Splittings of Classical Knots and
the Classification and Symmetries of Arborescent Knots. Preprint, 2010.

\bibitem{BoZi} M.~Boileau and H.~ Zieschang 
{\em Nombre de ponts et générateurs méridiens des entrelacs de Montesinos.}, Comment. Math. Helv., \textbf{60}(2): 270--279, 1985.

\bibitem{Buck} G.~Buck and J.~Simon. \newblock \href {https://doi.org/10.1016/S0166-8641(97)00211-3} {Thickness and crossing number of knots}.
\newblock {\em Topology and its Applications}, \textbf{91}(3): 245--257, 1999.

\bibitem{BZ} 
G.~Burde and H.~Zieschang {\em Knots}, Walter de gruyter, 2003.

\bibitem{Con67} J.~H.~Conway.
{\em An enumeration of knots and links, and some of their algebraic properties}, 
Computational Problems in Abstract Algebra (Proc. Conf., Oxford, 1967), 329--358, 1970.

\bibitem{cromwell2004}  P.~Cromwell {\em Knots and links}, Cambridge university press, 2004.

\bibitem{DEJvR}
Y.~Diao, C.~Ernst, and E.~J.~van Rensburg.
\newblock \href {http://dx.doi.org/10.1017/S0305004198003338} {Thicknesses of
  knots}.
\newblock {\em Math. Proc. Cambridge Philos. Soc.}, \textbf{126}(2): 293--310, 1999.

\bibitem{DEHL2018}
Y.~Diao, C.~Ernst, G.~Hetyei and P.~Liu.  
\newblock \href {https://arxiv.org/abs/1901.09778v1}
{{A Diagrammatic Approach for Determining the Braid Index of Alternating Links}}.
\newblock {\em ArXiv Mathematics e-prints}, 2018.

\bibitem{denne}
Elizabeth {Denne}.
\newblock \href {http://arxiv.org/abs/math/0510561} {{Alternating Quadrisecants
  of Knots}}.
\newblock {\em ArXiv Mathematics e-prints}, 2005.

\bibitem{gallotti}
R.~Gallotti and O.~Pierre-Louis.
\newblock \href {https://doi.org/10.1103/PhysRevE.75.031801} {Stiff knots}.
\newblock {\em Phys. Rev. E (3)}, \textbf{75}(3): 031801, 14, 2007.

\bibitem{GRvdM}
H.~Gerlach, {\relax Ph}.~Reiter, and H.~von~der Mosel.
\newblock \href {https://doi.org/10.1007/s00205-017-1100-9} {The elastic
  trefoil is the doubly covered circle}.
\newblock {\em Arch. Ration. Mech. Anal.}, 225(1): 89--139, 2017.

\bibitem{GiRvdM}
A.~Gilsbach, {\relax Ph}.~Reiter, and H.~von~der Mosel.
Symmetric elastic knots.
\newblock {\em ArXiv Mathematics e-prints}, 2021.

\bibitem{GM99}
O.~Gonzalez and J.~H. Maddocks.
\newblock \href {https://doi.org/10.1073/pnas.96.9.4769} {Global curvature,
  thickness, and the ideal shapes of knots}.
\newblock {\em Proc. Natl. Acad. Sci. USA}, \textbf{96}(9):4769--4773, 1999.

%
%
%
%


%
%
%

\bibitem{hoste1999knotscape} J.~Hoste and M.~Thistlethwaite. Knotscape, \href{https://www.math.utk.edu/~morwen/knotscape.html}{\url{https://www.math.utk.edu/~morwen/knotscape.html}}, 1999.

\bibitem{LS}
J.~Langer and D.~A. Singer.
\newblock \href {https://doi.org/10.1016/0040-9383(85)90046-1} {Curve
  straightening and a minimax argument for closed elastic curves}.
\newblock {\em Topology}, 24(1): 75--88, 1985.


\bibitem{knotinfo}
C. ~Livingston and A. H.~Moore, KnotInfo: Table of Knot Invariants, \href{http://www.indiana.edu/~knotinfo}{\url{http://www.indiana.edu/~knotinfo}}, 2020.

\bibitem{micheletti} M.~Marenda, E.~Orlandini, and C.~Micheletti. Discovering privileged topologies of molecular knots with self-assembling models. \emph{Nature Communications} \textbf9, 3051, 2018.%

\bibitem{milnor}
John~W. Milnor.
\newblock \href{http://dx.doi.org/10.2307/1969467} {On the total curvature of
  knots}.
\newblock {\em Ann. of Math. (2)}, \textbf{52}: 248--257, 1950.

\bibitem{Mu} K.~Murasugi
{\em On The Braid Index of Alternating Links}, Trans. Amer. Math. Soc. \textbf{326}: 237--260, 1991.

\bibitem{Mu2} K.~Murasugi, {\em Knot Theory \& Its Applications}, Birkh\"{a}user, Boston 1996.

\bibitem{rolfsen2003knots} D. ~Rolfsen, {\em Knots and links}, American Mathematical Soc. (346) 2003.

\bibitem{knotserver}
R.~Scharein.
\newblock \href {http://www.colab.sfu.ca/KnotPlot/KnotServer} {The knot
  server}, 2003.
\newblock Webpage, \verb|http://www.colab.sfu.ca/KnotPlot/KnotServer|, accessed 24 November 2017.

\bibitem{S2} H.~Schubert, \"Uber eine numerische Knoteninvariante, Math. Z. \textbf{61}: 245--288, 1954.


\bibitem{Sch2007} J.~Schultens, Bridge numbers of torus knots, Math. Proc. Camb. Phil. Soc. \textbf{143} (2007),
\href {http://dx.doi:10.1017/S0305004107000448}

%
%
%
%
 
 \bibitem{vdM:meek}
Heiko von~der Mosel.
\newblock \href {http://iospress.metapress.com/content/1WRYFLGDG3AL499J}
  {Minimizing the elastic energy of knots}.
\newblock {\em Asymptot. Anal.}, \textbf{18}(1-2): 49--65, 1998.

\bibitem{vdM:eke3}
Heiko von~der Mosel.
\newblock \href {http://dx.doi.org/10.1016/S0294-1449(99)80010-9} {Elastic
  knots in {E}uclidean {$3$}-space}.
\newblock {\em Ann. Inst. H. Poincar\'e Anal. Non Lin\'eaire}, \textbf{16}(2): 137--166,
  1999.
  
 \bibitem{Ya} S.~Yamada. The Minimal Number of Seifert Circles Equals The Braid Index of A Link. {\em Invent. Math.} \textbf{89}: 347--356, 1987.
  \end{thebibliography}
\end{document}